\documentclass[12pt]{amsart}
\usepackage[margin=35mm]{geometry}
\usepackage{amsmath,amssymb}            
\usepackage{hyperref}
\usepackage{mathtools}                  

\usepackage{todonotes}

\newcommand{\N}{\mathbb{N}}
\newcommand{\Z}{\mathbb{Z}}

\newcommand{\R}{\mathbb{R}}

\newcommand{\tO}{\mathtt{0}}
\newcommand{\tL}{\mathtt{1}}

\newcommand{\rb}[1]{\left( #1 \right)}
\newcommand{\abs}[1]{\left| #1 \right|}

\newtheorem{theorem}{Theorem}

\newtheorem{lemma}{Lemma}

\newtheorem{remark}{Remark}

\begin{document}
\title{M\"obius Orthogonality for the Zeckendorf Sum-of-Digits Function}

\author{Michael Drmota}
\email{michael.drmota@tuwien.ac.at}
\address{Institut f\"ur Diskrete Mathematik und Geometrie
TU Wien\\
Wiedner Hauptstr. 8--10\\
1040 Wien, Austria}
\author{Clemens M\"ullner}
\email{clemens.muellner@tuwien.ac.at}
\address{Institut f\"ur Diskrete Mathematik und Geometrie
TU Wien\\
Wiedner Hauptstr. 8--10\\
1040 Wien, Austria}
\author{Lukas Spiegelhofer}
\email{lukas.spiegelhofer@tuwien.ac.at}
\address{Institut f\"ur Diskrete Mathematik und Geometrie
TU Wien\\
Wiedner Hauptstr. 8--10\\
1040 Wien, Austria}
%

\date{\today} 
\thanks{All authors are supported by the
Austrian Science Foundation FWF, project F5502-N26, which is a part of the Special Research Program ``Quasi Monte Carlo Methods: Theory and Applications''.
Moreover, the authors want to acknowledge support by the project MuDeRa (Multiplicativity, Determinism and Randomness), 
which is a joint project between the ANR (Agence Nationale de la Recherche) and the FWF (Austrian Science Fund).
Furthermore, the authors want to thank Mariusz Lema\'nczyk for very helpful discussions.
}

\begin{abstract}
We show that the (morphic) sequence $(-1)^{s_\varphi(n)}$ is asymptotically orthogonal to all bounded multiplicative functions,
where $s_\varphi$ denotes the Zeckendorf sum-of-digits function.
In particular we have $\sum_{n<N} (-1)^{s_\varphi(n)} \mu(n) = o(N)$, that is, this sequence satisfies the Sarnak conjecture.
\end{abstract}

\maketitle

\section{Introduction}
The Sarnak conjecture \cite{Sarnak2011,Sarnak2012} says that the M\"obius function $\mu(n)$ is asymptotically orthogonal to any deterministic sequence, 
that is, for any sequence $(x_n)$
that can be realized in a deterministic flow we have:
\begin{equation}\label{eqMRP}
\sum_{n< N} \mu(n)x_n = o(N) \qquad (N\to\infty).
\end{equation}
If this equation holds true, we also say that $(x_n)$ satisfies a M\"obius randomness principle.
This conjecture has received a lot of attention during the last years and could be proved for several instances~\cite{Bourgain2013a, Bourgain2013, BSZ2013,
Davenport1937,DK2015, FM2015, Green2012, GT2012, EKL2016, HKLD2014, ELD2014, ELD2015, HLD2015, Karagulyan2015, KL2015, LS2015, Peckner2015, SU2015, Veech2016}.
In particular special automatic sequences were handled recently:~\cite{DT2005, DDM2015, Drmota2014, FKLM2016, Hanna2017, Katai2001, Katai1986, MR2010, MR2015}.
And finally the second author could solve the Sarnak conjecture for all automatic sequences \cite{Muellner2017}.

Automatic sequences constitute an interesting class of deterministic sequences that 
can be characterized in several different ways, see \cite{AS2003}.
For example they can be seen as codings of fixed points of morphisms (on sequences over a finite alphabet) of constant length.
The most prominent example is the Thue--Morse sequence 
$(t_n)_{n\in\mathbb{N}} = 
(\tO\tL\tL\tO\tL\tO\tO\tL\tL\tO\tO\tL\tO\tL\tL\tO\ldots)$ which
is the fixed point of the morphism $\sigma(\tO) = \tO\tL$, $\sigma(\tL) = \tL\tO$ starting with $\tO$ (where the coding is the identity).

It is, thus, a natural question whether a corresponding result holds for \emph{morphic sequences},
which are obtained by general morphisms, followed by a coding.
One of the simplest morphic sequences is the \emph{Fibonacci word} 
\[
(x_n)_{n\geq 1} = 
\left( {2+\left\lfloor {{n}\varphi }\right\rfloor -\left\lfloor {\left({n+1}\right)\varphi }\right\rfloor } \right)_{n\geq 1} = 
(\tO\tL\tO\tO\tL\tO\tL\tO\tO\tL\tO\tO\tL\tO\tL\tO\tO\tL\ldots)
\]
which is the fixed point of the morphism $\sigma(\tO)=\tO\tL$,
$\sigma(\tL)=\tO$, starting with $\tO$ (and where $\varphi = (1+\sqrt 5)/2$ denotes the golden mean).
A M\"obius Randomness Principle~\eqref{eqMRP} for this case, and more generally for Sturmian words, follows from~\cite[Theorem~5.2]{GT2012} by setting
$x_n = \lfloor n \alpha + \beta \rfloor - \lfloor (n-1) \alpha + \beta \rfloor$ or
$x_n = \lceil n \alpha + \beta \rceil - \lceil (n-1) \alpha + \beta \rceil$ for some irrational $\alpha$ and real $\beta$.
We note that Sturmian words are characterized as binary nonperiodic words having minimal factor complexity: 
there are exactly $k+1$ different factors (contiguous subsequences) of length $k$~\cite[Theorem~10.5.2]{AS2003}.
Automatic sequences, on the other hand, have sublinear factor complexity~\cite[Corollary~10.3.2]{AS2003}.
Moreover, morphic sequences have at most a quadratic number of factors of length $k$~\cite[Corollary~10.4.9]{AS2003}.
Sturmian, automatic and morphic sequences are, therefore, deterministic (the topological entropy being $0$).

Furthermore, we would like to mention a result by Houcein El Abdalaoui, Lemanczyk and De La Rue \cite{ELD2015},
which shows that automorphisms with quasi-discrete spectrum satisfy the Sarnak conjecture. 
This general result also covers some morphic sequences like the Fibonacci word.
However, the sequence $(-1)^{s_\varphi(n)}$ under consideration does not belong to this class \cite{Mariusz_private}.
Additionally, the Sarnak conjecture has been settled for some other morphic sequences 
by Ferenczi \cite{Ferenczi_private}, using different methods.

The purpose of this article is to settle the problem of M\"obius randomness for the sequence
\[
x_n = (-1)^{s_\varphi(n)},
\]
where $s_\varphi(n)$ denotes the Zeckendorf sum-of-digits of $n$, that is,
the minimal number of Fibonacci numbers needed to represent $n$ as their sum.
The sequence $s_\varphi\bmod 2$ is morphic, see~\cite[p. 14]{Bruyere}.
It is given by the following substitution $\sigma$ together with the coding $\pi$
\[
\sigma:\left\{\begin{array}{lll}
a&\mapsto&ab\\b&\mapsto&c\\c&\mapsto& cd\\d&\mapsto&a\end{array}\right\},
\qquad
\pi:\left\{\begin{array}{lll}a&\mapsto&\tO\\b&\mapsto&\tL\\c&\mapsto&\tL\\d&\mapsto&\tO\end{array}\right\},
\]
and we are interested in the fixed point starting with $a$.
Therefore,
\[
s_\varphi(n) \bmod 2=\left(\tO\tL\tL\tL\tO\tL\tO\tO\tL\tO\tO\tO\tL\tL\tO\tO\tO\tL\tO\tL\tL\tL\tO\tO\ldots\right).
\]

Actually, we prove a relation that is more general than (\ref{eqMRP}).
\begin{theorem}\label{Thmain}
Let $s_\varphi(n)$ be the Zeckendorf sum-of-digits function and $m(n)$ a bounded multiplicative function.
Then, we have
\begin{equation}\label{eqThmain}
\sum_{n<N} (-1)^{s_\varphi(n)} m(n) = o(N) \qquad (N\to\infty).
\end{equation}
\end{theorem}

The proof of Theorem~\ref{Thmain} is based on a general priciple that is due to 
Katai~\cite{Katai1986} (see also Bourgain, Sarnak and Ziegler~\cite{BSZ2013} for a quantitative version).
Suppose that $(x_n)_{n\in\mathbb{N}}$ is a bounded complex valued sequence with values in a finite set and that for every pair $(p,q)$ of different prime numbers we have
\[
\sum_{n< N} x_{pn} \overline {x_{qn}} = o(N).
\]
Then for all bounded multiplicative functions $m(n)$ it follows that 
\[
\sum_{n< N} x_n m(n) = o(N) \qquad (N\to\infty).
\]
Thus, it is sufficient to check the condition
\begin{equation}\label{eqmaincond}
\sum_{n< N} (-1)^{s_\varphi(pn)+s_\varphi(qn)} = o(N) \qquad (N\to\infty),
\end{equation}
which will be done in the main body of the paper.
We note that a statement like~\eqref{eqmaincond} was proved for the usual sum-of-digits function~\cite{Coquet1983,DT2005,Solinas1989}.

\begin{remark}
  It is not always clear if a morphic sequence (given by a substitution of non-constant length) is automatic or not
  (that is, whether it is  the fixed point of a substitution of constant length or not).
  In Section~\ref{sec4} we show that \eqref{eqmaincond} holds for all $q>p\geq 2$.
  In particular this implies the sequences $((-1)^{s_{\varphi}(k^{\lambda} n)})_{n\in \mathbb{N}}$
   are pairwisely different for all $k\geq 2$ and $\lambda \in \mathbb{N}$.
  Therefore, the $k$-kernel of $((-1)^{s_{\varphi}(n)})_{n\in\mathbb{N}}$ is infinite and, thus,  this sequence is not automatic 
  (see again~\cite{AS2003} for more details on this topic).
\end{remark}

In Section~\ref{sec2} we recall some facts about the Zeckendorf expansion of nonnegative integers and prove that there exists $n>0$
such that $s_\varphi(pn)+s_\varphi(qn)$ is odd.
In Section~\ref{sec3} we present a generating function approach for the analysis of sums of the form
\[ \sum_{F_{k-1}\le n< F_k} (-1)^{s_\varphi(pn)+s_\varphi(qn)}, \]
where $F_k$ denotes the $k$-th Fibonacci number.
With the help of these preliminaries the proof of (\ref{eqmaincond}) is then given in Section~\ref{sec4}.

\section{The Zeckendorf Sum-of-digits Function}\label{sec2}

For $k\geq 0$ let $F_k$ be the $k$-th Fibonacci number, that is,
$F_0=0$, $F_1=1$ and $F_k=F_{k-1}+F_{k-2}$ for $k\geq 2$.
By Zeckendorf's Theorem~\cite{Zeckendorf1972} every positive integer $n$ admits a unique representation
\[
n=\sum_{i\geq 2}\varepsilon_i F_i
,
\]
where $\varepsilon_i\in \{0,1\}$ and $\varepsilon_i=1\Rightarrow \varepsilon_{i+1}=0$.
By this theorem we may write the $i$-th coefficient $\varepsilon_i$ as a function of $n$.
The Zeckendorf sum-of-digits of $n$ is then defined as
\[
  s_\varphi(n) = \sum_{i\geq 2}\varepsilon_i(n).
\]
We set $s_\varphi(0)=0$.
We note that $s_\varphi(n)$ is the least $k$ such that $n$ is the sum of $k$ Fibonacci numbers.

The main purpose of this section is to prove that there exist integers $n',n''>0$ such that $s_\varphi(pn')+s_\varphi(qn')$ is even and $s_\varphi(pn'')+s_\varphi(qn'')$ is odd.
Actually we will prove a slightly more general property in Lemma~\ref{LeX}.

For the proof of Lemma~\ref{LeX} we need some preliminaries.
Let $\{x\}$ denote the fractional part of $x$ and $\langle x\rangle \coloneqq \{x + \frac{1}{2}\}-\frac{1}{2} \in [-\frac{1}{2},\frac{1}{2})$ 
denote the signed distance to the nearest integer.
Obviously, $\{x\} \equiv \langle x \rangle \bmod 1$ holds for all $x\in \R$.
We denote by $\varepsilon_k(n)$ the $n$-th digit of $n$ in the Zeckendorf expansion and, furthermore, 
\begin{align*}
  v(n,k) = \sum_{2 \leq i < k} \varepsilon_i(n) F_i.
\end{align*}
We define
\begin{align*}
  R_k(u) \coloneqq (-1)^k u \varphi + \left\{\begin{array}{cl} \vspace{2mm} \left[-\frac{1}{\varphi^{k-1}},\frac{1}{\varphi^k}\right), & 0 \leq u < F_{k-1};\\ 
				\left[-\frac{1}{\varphi^{k+1}},\frac{1}{\varphi^k}\right), & F_{k-1} \leq u < F_{k}, \end{array}\right.
\end{align*}
where $\varphi =  (\sqrt{5}+1)/2$.
This allows us to detect the last $k$ digits of $n$ in the Zeckendorf expansion
(see for example~\cite[Proposition~5.7]{Spiegelhofer2014}).

\begin{lemma}
  Let $k\geq 2, 0 \leq u < F_k$ and $n \geq 0 $. Then we have
  \begin{align*}
    v(n,k) = u
  \end{align*}
  if and only if
  \begin{align*}
    (-1)^k n \varphi \in R_k(u) + \Z.
  \end{align*}
\end{lemma}

We want to show that the functions $s_\varphi(pn)+s_\varphi(qn)$ have a quasi-additive property with respect to the Zeckendorf expansion (compare to~\cite{KW2016}).
We say that $n_1$ and $n_2$ are $r$-separated at position $k$ if 
$\varepsilon_i(n_1) = 0$ for $i\ge k-r$ and $\varepsilon_i(n_2) = 0$ for $i \le k+r$.
In particular this means that 
\[
\varepsilon_i(n_1+n_2) = \left\{ \begin{array}{ll}
\varepsilon_i(n_1) & \mbox{for $i < k-r$}, \\
0 & \mbox{for $k-r\le i \le k+r$}, \\
\varepsilon_i(n_2) & \mbox{for $i > k+r$}.
\end{array}\right.
\]
We define a shift operator $S: \mathbb{N}\to \mathbb{N}$ by
\begin{align*}
  S(n) = \sum_{k\geq 2} F_{k+1} \varepsilon_k(n).
\end{align*}

Furthermore, we say that a function $f(n)$ is quasi-additive with respect to the
Zeckendorf expansion if there exists $r\ge 0$ such that 
\[
f(n_1+n_2) = f(n_1) + f(n_2)
\]
for all integers $n_1,n_2$ that are $r$-separated at some position $k$ and 
\[
 f(n) = f(S(n))
\]
for all integers $n$ such that $v(n,r) = 0$.
Note that $s_\varphi$ is $0$-quasiadditive with respect to the Zeckendorf expansion.
\begin{lemma}\label{LeQA}
Suppose that  $q>p \geq 2$ are integers and that $f(n)=s_\varphi(pn)+s_\varphi(qn)$.
Then $f$ is quasi-additive with respect to the Zeckendorf expansion.
\end{lemma}

\begin{proof}
Since the sum of quasi-additive functions is quasi-additive,
it suffices to show the property for $n\mapsto s_\varphi(mn)$.
First, we need to find $r$ such that if $n_1$ and $n_2$ are $r$-separated, $mn_1$ and $mn_2$ are still $0$-separated, 
which yields the claim by the $0$-quasi-additivity of $s_\varphi$.

Choose $r$ in such a way that $\varphi^{r-1}>m$
and assume that $n_1$ and $n_2$ are $r$-separated at $k$.
Then $v(n_2,k+r+1)=0$. 
We obtain
\[\bigl\langle(-1)^{k+r+1} n_2 \varphi\bigr\rangle
\in\left[-\frac 1{\varphi^{k+r}},\frac 1{\varphi^{k+r+1}}\right)
\subseteq \left[-\frac 1{m\varphi^{k+1}},\frac 1{m\varphi^{k+1}}\right),
\]
therefore,
\[\bigl\langle(-1)^{k+1} mn_2 \varphi\bigr\rangle
=m(-1)^{r}\langle(-1)^{k+r+1} n_2 \varphi\bigr\rangle
\subseteq \left[-\frac 1{\varphi^{k}},\frac 1{\varphi^{k+1}}\right),
\]
from which it follows that $v(mn_2,k+1)=0$, that is, $\varepsilon_i(mn_2)=0$ for $i\leq k$.

Moreover, we have $\varepsilon_i(n_1)=0$ for $i\geq k-r$ by assumption.
This implies $n_1<F_{k-r}$, therefore, $mn_1<\varphi^{r-1}F_{k-r}\leq F_{k-r+r-1+1}=F_{k}$,
which implies $\varepsilon_i(mn_1)=0$ for $i\geq k$.
Therefore, $mn_1$ and $mn_2$ are $0$-separated at $k$.

A simple computation shows
\begin{align*}
  F_{k+1} = \varphi F_k - \rb{-\frac{1}{\varphi}}^{k},
\end{align*}
which implies
\begin{align*}
  S(n) = \varphi n - \sum_{k\geq 2} \varepsilon_k(n) \rb{-\frac{1}{\varphi}}^{k}.
\end{align*}
Thus, we find for $\varphi^{r-3} >m$ and $v(n,r) = 0$
\begin{align*}
  \abs{S(mn) - m S(n)} &= \abs{\varphi m n - \sum_{k\geq 2}\varepsilon_k(mn) \rb{-\frac{1}{\varphi}}^k - m \varphi n + m \sum_{k\geq 2}\varepsilon_k(n) \rb{-\frac{1}{\varphi}}^k}\\
    &\leq \sum_{k\geq 2} \varepsilon_k(mn) \frac{1}{\varphi^k} + m \sum_{k\geq r} \varepsilon_k(n) \frac{1}{\varphi^k}\\
    &\leq \frac{1}{\varphi^2} \frac{1}{1-\frac{1}{\varphi^2}} + m \frac{1}{\varphi^r} \frac{1}{1-\frac{1}{\varphi^2}}\\
    &= \frac{1}{\varphi} + m \frac{1}{\varphi^{r-1}} < \frac{1}{\varphi} + \frac{1}{\varphi^2} = 1.
\end{align*}
This gives $S(mn) = m S(n)$ for $\varphi^{r-3} > m$ and $v(n,r) = 0$ and by the $0$-quasi-additivity of $s_{\varphi}(.)$ 
\begin{align*}
  s_{\varphi}(m S(n)) = s_{\varphi}(S(mn)) = s_{\varphi}(mn),
\end{align*}
which shows that $n \mapsto s_{\varphi}(mn)$ is quasi-additive.
\end{proof}

\begin{lemma}\label{lem:multiplication_low_digits}
  Let $q>p \geq 1$ be integers and $n,k \in \N$ such that $v(n,k) = 0$ and $\varphi^{k-1} > 2 q$.
  Then we have
  \begin{align*}
    \max \{\ell : v(qn,\ell) = 0\} \leq \max \{\ell: v(pn,\ell) = 0\}.
  \end{align*}
\end{lemma}
\begin{proof}
  We find that $v(n,k) = 0$ implies 
  \begin{align*}
    \{(-1)^k n \varphi\} \in \left[-\frac{1}{\varphi^{k-1}},\frac{1}{\varphi^k}\right) + \Z \subseteq \left[-\frac{1}{2 q}, \frac{1}{2q}\right) + \Z.   
  \end{align*}
  Or, phrasing it in terms of the signed distance to the nearest integer,
  \begin{align*}
    \langle (-1)^k n \varphi\rangle \in \left[-\frac{1}{\varphi^{k-1}},\frac{1}{\varphi^k}\right) \subseteq \left[-\frac{1}{2 q}, \frac{1}{2q}\right).
  \end{align*}

  This gives
  \begin{align*}
    \langle (-1)^{\ell} q n \varphi\rangle &= (-1)^{k + \ell} q \cdot \langle(-1)^k n \varphi\rangle\\
    \langle (-1)^{\ell} p n \varphi\rangle &= (-1)^{k + \ell} p \cdot \langle(-1)^k n \varphi\rangle
  \end{align*}
  for all $\ell \geq 0$.

  Thus, we find that $v(qn,\ell) = 0$ if and only if
  \begin{align*}
    q (-1)^{k + \ell} \langle (-1)^k n \varphi\rangle \in \left[-\frac{1}{\varphi^{\ell-1}},\frac{1}{\varphi^{\ell}}\right).
  \end{align*}
  However, this implies
  \begin{align*}
    p (-1)^{k + \ell} \langle (-1)^k n \varphi\rangle \in \left[-\frac{1}{\varphi^{\ell-1}},\frac{1}{\varphi^{\ell}}\right)
  \end{align*}
  which is equivalent to $v(pn,\ell) = 0$.
\end{proof}

\begin{lemma}\label{LeX}
  For all $m \geq 2$ and integers  $q>p \geq 2$ there exists $r$ such that $n \mapsto s_{\varphi}(pn)$ and $n\mapsto s_{\varphi}(qn)$ are $r$-quasi-additive,
  and there exist positive integers $n',n''$ such that $v(n',r) = v(n'',r) = 0$,
  $n'$ and $n''$ can not be decomposed into two positive integers that are $r$-separated and
  \begin{align*}
    s_\varphi(qn')&\equiv s_\varphi(qn') \bmod m,\\
    s_\varphi(qn'') &\not \equiv s_\varphi(pn'') \bmod m.
  \end{align*}
\end{lemma}

\begin{proof}
By Lemma~\ref{LeQA} there exists a minimal $r_0$ such that $n\mapsto s_\varphi(qn)$ and $n\mapsto s_\varphi(pn)$ are $r$-quasiadditive for all $r\geq r_0$.
Therefore, the integers $n_i = F_{r_0+1+ i(2r_0+2)}$, for $i=0,\ldots,m-1$ are pairwise $r_0$-separated and we find $n_i = S^{i(r_0+2)}(n_0)$.
This gives by the $r_0$ quasiadditivity of $n\mapsto s_\varphi(qn)$ and $n\mapsto s_\varphi(pn)$,
  \begin{align*}
    s_{\varphi}(q(n_0+\ldots+n_{m-1})) \equiv m s_{\varphi}(q n_0) \equiv 0 \bmod m,\\
    s_{\varphi}(p(n_0+\ldots+n_{m-1})) \equiv m s_{\varphi}(p n_0) \equiv 0 \bmod m.
  \end{align*}
This concludes the proof of the first part, by choosing $r=r_0+1$ and $n' = n_0 + \ldots + n_{m-1}$.

  The main idea for the proof of the second statement is to find positive integers $n_1,n_2$ such that
  \begin{align*}
    s_\varphi(q(n_2+n_1)) + 1 &= s_\varphi(q n_2) + s_\varphi(q n_1),\\
    s_\varphi(p(n_2+n_1)) &= s_\varphi(p n_2) + s_\varphi(p n_1),
  \end{align*}
   and $v(n_1,r) = v(n_2,r) = v(n_2+n_1,r) = 0$, where $r = r_0+1$.
   This shows the existance of some $\tilde{n}$ such that
   \begin{align*}
      s_\varphi(q\tilde{n}) &\not \equiv s_\varphi(p\tilde{n}) \bmod m
   \end{align*}
   where $v(\tilde{n},r) = 0$. 
   We decompose $\tilde{n}$ into indecomposable parts and an argument by contradiction shows immediately that one of these parts has the desired property.
   
   It is sufficient to choose $k$ and $n_1,n_2$ such that 
   \begin{itemize}
    \item  $v(qn_2,k+4) = F_{k+1}$ and $v(pn_2,k+1) = 0$
    \item  $F_k \leq q n_1 < F_{k+1}$ and $p n_1 < F_k$:
   \end{itemize}
   We find that $\varepsilon_i(pn) = \varepsilon_i(p n_2) + \varepsilon_i(p n_1)$ as $pn_2$ and $pn_1$ are $0$ separated at position $k+1$.
   Furthermore, we find $\varepsilon_i(qn) = \varepsilon_i(q n_2) + \varepsilon_i(q n_1)$ for all $i \in \mathbb{N} \setminus\{k,k+1,k+2,k+3\}$ and the digits at position $k+3,k+2,k+1,k$ 
   are $(0,1,0,0)$, $(0,0,1,0)$ and $(0,0,0,1)$ for $qn, qn_2$ and $qn_1$ respectively.\\
   For parity reasons, we ask for $v(qn_2,k+5) = F_{k+1}$ instead of $v(q n_2, k+4) = F_{k+1}$.
  
  Since $(F_n)_{n\in \mathbb{N}}$ is periodic modulo $q$, we can choose $n_1' > 2 F_r p$ and $k\in \mathbb{N}$ with $\varphi^{k} > 2 q$ such that $n_1' \cdot q = F_k$.
  We choose $n_1 = n_1' + n_1''$ such that $v(n_1,r) = 0$ and $n_1'' < F_r$.
  The condition $n_1' > 2 F_r p$ assures that $F_k \leq q n_1 < F_{k+1}$ and $p n_1 < F_k$.
  
  We use again the following identity
  \begin{align*}
    \varphi F_k = F_{k+1} - \left( -\frac{1}{\varphi} \right)^{k}.
  \end{align*}
  Thus, we find
  \begin{align*}
    R_{k+5}(F_{k+1}) + \Z  &= -\frac{1}{\varphi^{k+1}} + \left[-\frac{1}{\varphi^{k+4}},\frac{1}{\varphi^{k+5}}\right) + \Z\\
&\subseteq \left[-\frac{1}{\varphi^{k}},\frac{1}{\varphi^{k+1}}\right) + \Z = R_{k+1}(0) + \Z .
  \end{align*}

  Let $I$ be the representative of $R_{k+5}(F_{k+1}) \bmod 1$ in $[-\frac{1}{2},\frac{1}{2})$.
  Thus we find $I \subset \left[-\frac{1}{\varphi^{k}},\frac{1}{\varphi^{k+1}}\right)$.
  We denote by $I' \coloneqq I \cdot \frac{1}{q} \subset \left[-\frac{1}{q \varphi^k}, \frac{1}{q \varphi^{k+1}}\right)$.
  As $\{n \varphi\}$ is dense in $[0,1)$, we find $n_2$ such that $\langle n_2 \varphi \rangle \in I' \subset R_{k+1}(0)$.\\
  By construction, we find $\langle q n_2 \varphi \rangle \in I$ and $v(n_2,k+1) = 0$.
  $\langle q n_2 \varphi \rangle \in I$ gives $v(q n_2, k+5) = F_{k+1}$ and, therefore, $v(q n_2, k+4) = F_{k+1}$.
  Since $v(n_2,k+1) = 0$ and $\varphi^k > 2q$ we can apply Lemma~\ref{lem:multiplication_low_digits}, which gives $v(p n_2,k+1) = 0$ as required.
\end{proof}

\section{Generating Functions}\label{sec3}
In this section, let $q>p \geq 2$ be integers and let $f(n)$ be defined by
\begin{equation}\label{eqDeffn}
f(n) = s_\varphi(pn) + s_\varphi(qn).
\end{equation}
By Lemma~\ref{LeQA} the function $f$ is quasi-additive.

Next we fix a finite subset $L = \{\ell_1, \ell_2, \ldots, \ell_d\}$ of positive integers with the property
$\ell_1 \le 2r$ and $\ell_j + 1 < \ell_{j+1} \le \ell_j + 2r +1$, $1\le j < d$, and consider the generating function
\[
H_L(x,z) = \sum_{k\ge 3} x^k \sum_{F_{k-1} \le n < F_k} z^{f(n + N_L(k))},
\]
where $N_L(k)$ is given by
\[
N_L(k) = \sum_{\ell\in L} F_{k+\ell} = \sum_{j=1}^d F_{k+\ell_j}.
\]
Furthermore, let $\mathcal{B}'$ be the set of positive integers $n$ 
that have no decomposition of the form $n= n_1+n_2$, where $n_1$ and $n_2$ are non-zero and
$r$-separated (at some position $k$) and $\mathcal{B} = \{n \in \mathcal{B}': v(n,r+1) = F_r\}$. Then we set
\[
B(x,z) = \sum_{n\in \mathcal{B}} x^{\ell(n)} z^{f(n)}
\]
and 
\[
B_L(x,z) = \sum_{n\in \mathcal{B}} x^{\ell(n)} z^{f(n + N_L(\ell(n)))},
\]
where $\ell(n) = k$ if $F_{k-1} \le n < F_k$.
The generating functions $B'(x,z)$ and $B_L'(x,z)$ are defined analogously.
The generating function $H(x,z)$ can be expressed in the following form
(compare with \cite{KW2016}).

\begin{lemma}\label{LeHLrep}
Suppose that $f(n)$ is quasi-additive with respect to the
Zeckendorf expansion for some integer $r\ge 1$. Then we have
\[
H_L(x,z) = B_L'(x,z) + \frac{B'(x,z)B_L(x,z)}{1-x-x^{r+1} B(x,z)}.
\]
\end{lemma}

\begin{proof}
The generating function $H_L$ is a functional of all Zeckendorf expansions of 
$n + N_L(\ell(n))$,
where $n$ runs over all positive integers, that is,
we always add the digits corresponding to $N_L(\ell(n))$.
The expansion of every $n \notin \mathcal{B}'$ can be decomposed in the following way.
They start with an expansion corresponding to an element in $\mathcal{B}'$ (which may start with less than $r$ zeros).
Then there is a possibly empty sequence of pairs consisting of a sequence of at least $r+1$ zeros followed by an expansion corresponding to an element in $\mathcal{B}$ 
(which start with exactly $r$ zeros).
Finally there is an element of $\mathcal{B}$ followed by digits corresponding to $L$.
This means that $n = n_0 + S^{k_1}(n_1) + \ldots + S^{k_t}(n_t)$ where $n_0 \in \mathcal{B}', n_1,\ldots,n_{t} \in \mathcal{B}, k_i \geq r+1$ and the $S^{k_i}(n_i)$
are pairwise $r$-separated and by the quasi-additivity of $f$: $f(n+ N_L(\ell(n))) = f(n_0) + f(n_1) + \ldots + f(n_{t-1}) + f(n_{t} + N_L(\ell(n_{t})))$.

When $n$ already belongs to $\mathcal{B}'$, we do not decompose it at all.
This gives,
\begin{align*}
H_L(x,z) &= B_L'(x,z) + B'(x,z) \frac 1{1- B(x,z) \frac{x^{r+1}}{1-x} } B_L(x,z)\\
  &= B_L'(x,z) + \frac{B'(x,z) B_L(x,z)}{1-x-x^{r+1} B(x,z)}.
\end{align*}
This proves the lemma.
\end{proof}

The next lemma gives a quantitative bound on how many elements there are in $\mathcal{B}$.

\begin{lemma}\label{LeB}
Assume that $r\geq 2$ and let $1<\varphi_r< \varphi$ be the solution of the equation
\[
1 - \frac 1{\varphi_r} - \frac 1{\varphi_r^2} + \frac 1{\varphi_r^{2r+2}} = 0.
\]
Then we have
\begin{align*}
\# \{ n \in \mathcal{B}: \ell(n) = k \} = O\left( \varphi_r^k \right),\\
\# \{ n \in \mathcal{B}': \ell(n) = k \} = O\rb{\varphi_r^k}.
\end{align*}
\end{lemma}

\begin{proof}
We only prove the statement for $\mathcal{B}'$ as $\mathcal{B} \subseteq \mathcal{B}'$.
We are interested in the number of $\tO$-$\tL$-sequences of length $k-2$, starting with $\tL$, with the
property that two adjacent $\tL$s are separated by at least one zero but at most $2r$ zeros.
These sequences can be seen as (possibly empty) concatenations of the words $\tL\tO,\tL\tO\tO,\ldots,\tL\tO^{2r}$ followed by $\tL$ and a (possible empty) concatenation of zeros.
The corresponding generating function is $\frac {x^2}{1-(x^2+\cdots+x^{2r+1})} \frac{x}{1-x}=\frac{x^2-x^3}{1-x-x^2+x^{2r+2}} \frac{x}{1-x}$.
The denominator $1-(x^2+\cdots+x^{2r+1})$ is positive at $x=1/\varphi$ (since $1-1/\varphi^2-1/\varphi^3-\cdots=0$) and negative at $x=1$, 
therefore, there is a unique zero $1/\varphi_r$ in the range $(1/\varphi,1)$, which yields the claim by Pringsheim's Theorem.
\end{proof}

It is an immediate consequence of Lemma~\ref{LeB} that $B(x,z),B_L(x,z), B'(x,z)$ and $B_L'(x,z)$
converge absolutely for all $x,z$ with $\lvert x\rvert < 1/\varphi_r$ and $\lvert z\rvert = 1$.

\section{Proof of Theorem~\ref{Thmain}}\label{sec4}

As mentioned in the introduction it is sufficient to show (\ref{eqmaincond}) which we
will do now in several steps. As above we fix two integers $q>p \geq 1$.

\begin{lemma}\label{Lesec4.1}
There exist $\eta> 0$ such that uniformly for all finite $L$ (as described above)
\[
\sum_{F_{k-1}\le n < F_k} (-1)^{f(n+N_L(k))} = O\left( \varphi^{(1-\eta) k} \right).
\]
\end{lemma}

\begin{proof}
We first note that the above sum is just the coefficient of $x^k$ of the function $H_L(x,-1)$.
Thus, we have to show that $H_L(x,-1)$ has no singularities for $\lvert x\rvert \leq 1/\varphi+\varepsilon$ for some $\varepsilon>0$ and apply Cauchy's formula.
For this purpose we use the representation of $H_L(x,z)$ given in Lemma~\ref{LeHLrep}, where $r$ will be chosen later.
As mentioned at the end of Section~\ref{sec3} (after the proof of Lemma~\ref{LeB}) the functions $B(x,-1)$ and $B_L(x,-1)$ have no singularities in the region $\lvert x\rvert< 1/\varphi_r$.
Hence the only possible singularity of $H_L(x,-1)$ in the region $\lvert x\rvert <1/\varphi_r$ could be due to a solution of the equation 
\begin{equation}\label{eqzero}
x+x^{r+1} B(x,-1) = 1.
\end{equation}
We show that there exists $x_1>x_0 := 1/\varphi$ such that this equation does not have a solution for $\lvert x\rvert\leq x_1$, which implies that $H_L(x,-1)$ is uniformly bounded for $\lvert x\rvert\leq x_1$.
We first note that 
\[
x_0 + x_0^{r+1} B(x_0,1) = 1
\]
since the coefficient of $x^k$ of $H_L(x,1)$ is of order $\varphi^{k}$ and, therefore, it is necessary that $H_L(x,1)$ has a singularity at $x=x_0$.

By Lemma~\ref{LeX} there are integers $n', n''$ and $r$ for which $f(n')$ is even and $f(n'')$ is odd and $n',n''\in\mathcal B$.
It follows that there exists $x_1'>x_0$ such that
\begin{equation}\label{eqtoshow}
|B(x,-1)| < B(x_0,1)  \qquad \mbox{for all $|x|\le x_1'$.}
\end{equation}
This further implies that there exists $x_1>x_0$ such that for $\lvert x\rvert\leq x_1$
\begin{align*}
\lvert x+x^{r+1}B(x,-1)\rvert&\leq \lvert x\rvert+\lvert x\rvert^{r+1}\lvert B(x,-1)\rvert\\
&\leq x_1+x_1^{r+1}\lvert B(x,-1)\rvert\\
&< x_0+x_0^{r+1}B(x_0,1)=1.
\end{align*}
Finally, by Cauchy's formula we obtain
\begin{align*}
\left|\sum_{F_{k-1}\le n < F_k} (-1)^{f(n+N_L(\ell(n)))} \right|
&= \left| \frac 1{2\pi i} \int_{|x|= x_1} H_L(x,-1) x^{-k-1}\, dx \right| \\
&\le \max_{|x| = x_1} |H_L(x,-1)| x_1^{-k} \\
&= O\left( \varphi^{(1-\eta) k} \right),
\end{align*}
where $\eta> 0$ is defined by $\varphi^{1-\eta} = x_1^{-1}$.
\end{proof}

With the help of Lemma~\ref{Lesec4.1} we can finally derive the desired upper bound.

\begin{lemma}\label{Lesec4.2}
There exists $\eta> 0$ such that 
\[
\sum_{1 \le n < N} (-1)^{f(n)} = O\left( N^{1-\eta} \right).
\]
\end{lemma}

\begin{proof}
First we observe that Lemma~\ref{Lesec4.1} implies that we also have 
\begin{equation}\label{eqX}
\sum_{1\leq n < F_k} (-1)^{f(n+N_L(k))} = O\left( \varphi^{(1-\eta) k} \right).
\end{equation}
uniformly for all finite sets $L = \{\ell_1, \ell_2, \ldots, \ell_d\}$ of positive integers with the property
$\ell_1 \le 2r$ and $\ell_j + 1 < \ell_{j+1}$, $1\le j < d$.

To see this, assume first that there is some $j$ with $\ell_{j+1} > \ell_j + 2r+1$.
In this case we can apply Lemma~\ref{LeQA} and split the contribution of $N_L(k)$ into at least two parts so that we can restrict ourselves to the situation covered by Lemma~\ref{Lesec4.1}.
Secondly, we partition the sum over $n< F_k$ into several subsums:
\[
\sum_{n < F_k} (-1)^{f(n+N_L(k))} = \sum_{j\le k} \sum_{F_{j-1}\le n < F_j} (-1)^{f(n+N_L(k))}.
\]
For $j> k - 2r-1$ we can directly apply Lemma~\ref{Lesec4.1} by shifting the corresponding sets $L$ accordingly.
Finally for $j\le k-2r-1$ we apply first Lemma~\ref{LeQA} and separate the contribution of $f(n + N_L(k))$ into $f(n) + f(N_L(k))$ and apply Lemma~\ref{Lesec4.1} for $L = \emptyset$.
This gives
\[
\sum_{n < F_k} (-1)^{f(n+N_L(k))} = O\left( \sum_{j\le k} \varphi^{(1-\eta) j} \right) =
O\left( \varphi^{(1-\eta) k} \right). 
\]

In order to complete the proof of Lemma~\ref{Lesec4.2} we consider the Zeckendorf expansion of $N$:
\[
N = \sum_{j=1}^J  F_{k_j},
\]
where $k_j > k_{j+1} +1$ (for $1\le j< J$). Furthermore we set $L_1 = \emptyset$ and
$L_j = \{ k_{j-1}-k_j, k_{j-2}-k_j,\ldots, k_1-k_j \}$ for $2\le j \le J$. Then we have
\[
\sum_{n< N} (-1)^{f(n)} = \sum_{j=1}^J \sum_{n < F_{k_j}} (-1)^{f(n + N_{L_j}(k_j))}.
\]
Since (\ref{eqX}) holds uniformly for all possible finite sets $L$ we thus obtain
\[
\sum_{n< N} (-1)^{f(n)} = O\left( F_{k_1}^{1-\eta} \right) = O\left( N^{1-\eta} \right),
\]
which completes the proof of the lemma.
\end{proof}

\section{Possible extensions}
We proved that the infinite word $s_\varphi$ modulo $2$ satisfies a M\"obius Randomness Principle (MRP).
This gives rise to several questions on possible extensions of this result. 
For example, is it true that every sequence observed by the symbolic dynamical system associated with $s_\varphi$ modulo $2$ satisfies the Sarnak conjecture?
Moreover, it is certainly possible to extend our results to the case $s_\varphi$ modulo $m$ for arbitrary $m\geq 2$.

Finally, we want to mention some possible directions of further research:
\begin{enumerate}
\item Prove a MRP for $(-1)^{s_\alpha(n)}$, where $s_\alpha$ is the Ostrowski sum-of-digits function (for arbitrary irrational $\alpha$), 
generalizing the case $\alpha=\varphi$ (see~\cite{Berthe2001} for a survey on the Ostrowski numeration system). 
We expect that the methods of this paper can be used at least for the case of quadratic irrational $\alpha$.
\item Prove a MRP for sequences that are automatic with respect to the Zeckendorf numeration (they appeared under the name of Fibonacci-automatic sequences in \cite{fibonacci_automatic}).
It is expected that one can find a similar decomposition for these automata as in \cite{Muellner2017} and use the techniques that were developed in this paper.
\item Prove a MRP for general morphic sequences, that is, sequences that are projections of a fixed point of a morphism.
	This is probably the most difficult generalization and will likely need some new ideas.
	However, the methods used in this paper seem to fit this framework rather well.

\end{enumerate}

\bibliographystyle{plain}
\bibliography{moebius}





\end{document}